\documentclass[a4paper,12pt]{article}
\usepackage{latexsym}	
\usepackage{amsmath}
\usepackage{amsthm}	

\usepackage{graphicx}
\usepackage{txfonts}
\usepackage{bm}
\usepackage{color}
\usepackage{epic,eepic}

\newcommand{\RED}[1]{{\color{red}#1}} 
 \renewcommand{\RED}[1]{{#1}} 
\newcommand{\BLU}[1]{{\color{blue}#1}} 
 \renewcommand{\BLU}[1]{{#1}} 

\usepackage{geometry}
\numberwithin{equation}{section}

\newtheorem{theorem}{Theorem}[section]
\newtheorem{proposition}[theorem]{Proposition}
\newtheorem{corollary}[theorem]{Corollary}
\newtheorem{lemma}[theorem]{Lemma}
\newtheorem{remark}{Remark}[section]
\newtheorem{example}{Example}[section]

\newcommand{\memo}[1]{{\bf \small \RED{[MEMO:}} \BLU{\small #1} \ {\bf \small \RED{:end]} }}  
   \renewcommand{\memo}[1]{}           
\newcommand{\OMIT}[1]{{\bf [OMIT:} #1 \ {\bf --- end OMIT] }}  
   \renewcommand{\OMIT}[1]{}            

\newcommand{\iffS}{\ \Leftrightarrow\ }

\newcommand{\RR}{{\mathbb{R}}}
\newcommand{\ZZ}{{\mathbb{Z}}}

\newcommand{\unitvec}[1]{\bm{1}\sp{#1}}
\newcommand{\Lnat}{{L$^{\natural}$}}
\newcommand{\Mnat}{{M$^{\natural}$}}

\newcommand{\finbox}{\hspace*{\fill}$\rule{0.2cm}{0.2cm}$}
\newcommand{\todaye}{\the\year/\the\month/\the\day}

\begin{document}

\title{
Shapley--Folkman-type Theorem
\\
for Integrally Convex Sets
}

\author{
Kazuo Murota%
\thanks{
The Institute of Statistical Mathematics,
Tokyo 190-8562, Japan; 
and
Faculty of Economics and Business Administration,
Tokyo Metropolitan University, 
Tokyo 192-0397, Japan,
murota@tmu.ac.jp}
\ and 
Akihisa Tamura%
\thanks{Department of Mathematics, Keio University, 
Yokohama 223-8522, Japan,
aki-tamura@math.keio.ac.jp}
}

\date{May 2023 / August 2024}

\maketitle

\begin{abstract}
The Shapley--Folkman theorem is a statement  
about the Minkowski sum of (non-convex) sets,
expressing the closeness of the Minkowski sum 
to convexity in a quantitative manner.
This paper establishes similar theorems
for integrally convex sets, \Mnat-convex sets, and \Lnat-convex sets,
which are major classes of discrete convex sets in discrete convex analysis.
\end{abstract}

\Lnat-convex set,
\Mnat-convex set,
Minkowski sum,
Shapley--Folkman theorem

\tableofcontents

\newpage



\section{Introduction}
\label{SCintro}

The Shapley--Folkman theorem
is concerned with the Minkowski sum of (non-convex) sets
and expresses the closeness of the Minkowski sum 
to convexity in a quantitative manner.
The theorem was first discovered in the literature of economics
(Arrow--Hahn \cite{AH71}, 
Starr \cite{Sta69,Sta08})
and found applications also in optimization
(Aubin--Ekeland \cite{AE76},
Ekeland--T{\'e}mam \cite{ET99},
Bertsekas \cite{Ber09,Ber16})
and other fields of mathematics
(Fradelizi--Madiman--Marsiglietti--Zvavitch \cite{FMMZ18}).

To describe the Shapley--Folkman theorem 
we need to introduce some terminology and notation.
The {\em Minkowski sum} (or {\em vector sum}) 
of sets $S_{1}$, $S_{2} \subseteq \RR\sp{n}$ 
means the subset of $\RR\sp{n}$ 
defined by
\begin{equation} \label{minkowsumGdef}
S_{1}+S_{2} = \{ x + y \mid x \in S_{1}, \  y \in S_{2} \} .
\end{equation}
This operation can natually be extended to the
Minkowski sum 
$\sum_{i=1}\sp{m} S_{i} = S_{1}+ S_{2} + \cdots + S_{m}$
of an arbitrary number of sets
$S_{i} \subseteq \RR\sp{n}$ $(i=1,2,\ldots,m)$. 
The Minkowski sum of convex sets is again convex.
For any subset $S$ of $\RR\sp{n}$, we denote its {\em convex hull} by 
$\overline{S}$, which is, by definition, 
the smallest convex set containing $S$.
As is well known, $\overline{S}$ coincides with the set of 
all convex combinations of (finitely many) elements of $S$.
It is known that 
$\overline{S_{1}+ S_{2} + \cdots + S_{m}}
 = \overline{S_{1}} + \overline{S_{2}} + \cdots + \overline{S_{m}}$.

For any set $S$ $(\subseteq \RR\sp{n})$,
the {\em radius} ${\rm rad}(S)$
and the {\em inner radius} $r(S)$
are defined by  
\begin{align}
{\rm rad}(S) &= \inf_{x \in \RR\sp{n}} \sup_{y \in S} \| x - y \|_{2},
\label{SFraddef} \\
r(S) &= \sup_{x \in \overline{S}} 
\inf_{T}\{ {\rm rad}(T) \mid T \subseteq S, x \in \overline{T} \}.
\label{SFinnerraddef} 
\end{align}
The inner radius $r(S)$ expresses the
size of holes or dents in $S$, and 
we have $r(S)=0$ for a convex set $S$.

The following theorem 
\cite[Theorem B.10]{AH71}
expresses the closeness of the Minkowski sum 
of (non-convex) sets to convexity in a quantitative manner.
This theorem is often referred to as the 
{Shapley--Folkman--Starr theorem},
as it was derived by Starr \cite{Sta69}
from the Shapley--Folkman theorem \cite[Theorem B.9]{AH71}
as a (non-trivial) corollary.

\begin{theorem}[Shapley--Folkman--Starr]  \label{THshfostarr}
Let $S_{i}$ $(i=1,2,\ldots,m)$ be compact subsets of $\RR\sp{n}$
such that $r(S_{i}) \leq L$ for $i=1,2,\ldots,m$ for some $L \in \RR$.
Let $W = S_{1} + S_{2} + \cdots + S_{m}$.
For any $x \in \overline{W}$,
there exists $z \in W$ that satisfies 
$\| x - z \|_{2} \leq L \sqrt{\min (n,m)}$.
\finbox
\end{theorem}

A key fact used in the proof of Theorem~\ref{THshfostarr}
is the following theorem,
which formulates a phenomenon in the Minkowski summation
that may be compared to Carath{\' e}odory's theorem
for convex combinations.

\begin{theorem}[Shapley--Folkman]  \label{THshapfolkG0}
Let $S_{i} \subseteq \RR\sp{n}$  $(i=1,2,\ldots,m)$,
and $W = S_{1} + S_{2} + \cdots + S_{m}$.
For any $x \in \overline{W}$,
there exists a subset $I$ of the index set $\{ 1,2,\ldots,m \}$
such that $|I| \leq \min(n,m)$ and
$x \in \overline{\sum_{i \in I} S_{i}} + \sum_{j \in J} S_{j}$,
where $J = \{ 1,2,\ldots,m \} \setminus I$.
\finbox
\end{theorem}

Theorem~\ref{THshapfolkG0} is
ascribed to Shapley and Folkman in \cite[Theorem B.8]{AH71},
and is often referred to as the Shapley--Folkman lemma.
Although the statement of \cite[Theorem B.8]{AH71}
involves an assumption of compactness of each $S_{i}$,
it is possible to avoid this assumption 
by using an algebraic proof based on Carath{\' e}odory's theorem
(Bertsekas \cite[Proposition 5.7.1]{Ber09}, 
Fradelizi--Madiman--Marsiglietti--Zvavitch \cite[Lemma 2.3]{FMMZ18}).
Alternative proofs of Theorem~\ref{THshapfolkG0} 
can be found in Ekeland--T{\'e}mam \cite[Appendix I]{ET99} 
(without  the compactness assumption)
and Howe \cite{How12} (under the compactness assumption).

\medskip

The objective of this paper is to 
establish theorems similar to Theorem~\ref{THshfostarr}
in the context of discrete convex analysis 
\cite{Fuj05book,Mdca98,Mdcasiam,Mdcaeco16}.
Section~\ref{SCprelimDCA} is devoted to the preliminaries 
from discrete convex analysis,
and the main results are described in Section~\ref{SCresult}.
Theorems \ref{THshapfolkLinfIC} and \ref{THshapfolkL2IC} 
give two variants of the Shapley--Folkman-type theorem
for integrally convex sets,
Theorem~\ref{THshapfolkM} deals with \Mnat-convex sets,
and Theorem~\ref{THshapfolkL}  with \Lnat-convex sets
(where $\natural$ is read ``natural'').
The proofs are given in Section~\ref{SCproof},
where Theorems \ref{THshfostarr} and \ref{THshapfolkG0} are used.

\section{Preliminaries from Discrete Convex Analysis}
\label{SCprelimDCA}

\subsection{Integrally convex sets}
\label{SCintconvsetdef}

Integral convexity is a fundamental concept 
in discrete convex analysis,
introduced by Favati--Tardella \cite{FT90} for functions
defined on the integer lattice $\ZZ\sp{n}$.  
In this paper we use the concept of
integrally convex sets, as formulated in 
\cite[Section 3.4]{Mdcasiam} as a special case of 
integrally convex functions.
The reader is referred to \cite{MT23ICsurv}
for a recent comprehensive survey on integral convexity.

For $x \in \RR^{n}$ the {\em integral neighborhood} of $x$ is defined by
\begin{equation} \label{Nxdef}
{\rm N}(x) = \{ z \in \mathbb{Z}^{n} \mid | x_{i} - z_{i} | < 1 \ (i=1,2,\ldots,n)  \} .
\end{equation}
It is noted that 
strict inequality ``\,$<$\,'' is used in this definition
and ${\rm N}(x)$ admits an alternative expression
\begin{equation}  \label{intneighbordeffloorceil}
{\rm N}(x) = \{ z \in \ZZ\sp{n} \mid
\lfloor x_{i} \rfloor \leq  z_{i} \leq \lceil x_{i} \rceil  \ \ (i=1,2,\ldots, n) \} ,
\end{equation}
where, for $t \in \RR$ in general, 
$\left\lfloor  t  \right\rfloor$
denotes the largest integer not larger than $t$
(rounding-down to the nearest integer) and 
$\left\lceil  t   \right\rceil$ 
is the smallest integer not smaller than $t$
(rounding-up to the nearest integer).
That is,
${\rm N}(x)$ consists of all integer vectors $z$ 
between
$\left\lfloor x \right\rfloor 
=( \left\lfloor x_{1} \right\rfloor ,\left\lfloor x_{2} \right\rfloor , 
  \ldots, \left\lfloor x_{n} \right\rfloor)$ 
and 
$\left\lceil x \right\rceil 
= ( \left\lceil x_{1} \right\rceil, \left\lceil x_{2} \right\rceil,
   \ldots, \left\lceil x_{n} \right\rceil)$.

For a set $S \subseteq \ZZ^{n}$
and $x \in \RR^{n}$
we call the convex hull of $S \cap {\rm N}(x)$ 
the {\em local convex hull} of $S$ around $x$.
A nonempty set
$S \subseteq \ZZ^{n}$ is said to be 
{\em integrally convex} if
the union of the local convex hulls 
$\overline{S \cap {\rm N}(x)}$ over $x \in \RR^{n}$ 
is convex.
In other words, a set $S \subseteq \ZZ^{n}$ is called 
integrally convex if
\begin{equation}  \label{icsetdef0}
 \overline{S} = \bigcup_{x \in \RR\sp{n}} \overline{S \cap {\rm N}(x)}.
\end{equation}
This condition is equivalent to saying that
every point $x$
in the convex hull of $S$ is contained in the convex hull of 
$S \cap {\rm N}(x)$, i.e.,
\begin{equation}  \label{icsetdef1}
 x \in \overline{S} \ \Longrightarrow  x \in  \overline{S \cap {\rm N}(x)} .
\end{equation}
Obviously, every subset of $\{ 0, 1\}\sp{n}$ is integrally convex.

We say that a set $S \subseteq \ZZ\sp{n}$ is
{\em hole-free} if
\begin{equation}  \label{icsetholefree}
  S = \overline{S} \cap \ZZ\sp{n} .
\end{equation}
It is known that an integrally convex set is hole-free;
see \cite[Proposition~2.2]{MT23ICsurv} for a formal proof.

\subsection{Minkowski sum in discrete convex analysis}

Minkowski summation is an intriguing operation in discrete setting.
The naive looking relation
\begin{equation} \label{convminkowsumG}
   S_{1}+S_{2}  = ( \overline{S_{1}+ S_{2}}) \cap \ZZ\sp{n} 
\end{equation}
is not always true,
as Example~\ref{EXicdim2sumhole} below shows.
It may be said that 
if \eqref{convminkowsumG} is true
for some class of discrete convex sets, 
this equality captures a certain essence of the discrete convexity in question.

\begin{example}[{\cite[Example 3.15]{Mdcasiam}}] \rm \label{EXicdim2sumhole}
The Minkowski sum of
$S_{1} = \{ (0,0), (1,1) \}$
and
$S_{2} = \{ (1,0), (0,1) \}$
is equal to 
$S_{1}+S_{2} = \{ (1,0), (0,1), (2,1), (1,2) \}$,
for which
$(1,1) \in (\overline{S_{1}+S_{2}}) \setminus (S_{1}+S_{2})$.
That is, the Minkowski sum $S_{1}+S_{2}$ has a `hole' at $(1,1)$.
See Figure~\ref{FGminkowhole}.
\finbox
\end{example}

\begin{figure}
\centering
 \includegraphics[height=30mm]{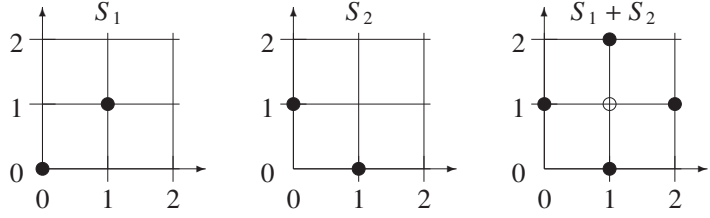}
\caption{Minkowski sum of discrete sets}
\label{FGminkowhole}
\end{figure}

In Example~\ref{EXicdim2sumhole} above,
both $S_{1}$ and $S_{2}$ are integrally convex.
This shows that \eqref{convminkowsumG} 
is not guaranteed for integrally convex sets
and that the Minkowski sum of integrally convex sets
is not necessarily integrally convex.

A subclass of integrally convex sets,
called \Mnat-convex sets,
is well-behaved with respect to Minkowski summation.
A (nonempty) set $S \subseteq \ZZ\sp{n}$
is called {\em \Mnat-convex}
if it enjoys the following exchange property:
\begin{quote}
For any $x, y \in S$ and 
$i \in \{ 1,2,\ldots, n \}$ with $x_{i} > y_{i}$, 
we have
\\
(i)
$x -\unitvec{i} \in S$, $y+\unitvec{i} \in S$
\ or \  
\\
(ii) there exists some $j \in \{ 1,2,\ldots, n \}$ such that
$x_{j} < y_{j}$,
$x-\unitvec{i}+\unitvec{j}  \in S$, and $y+\unitvec{i}-\unitvec{j} \in S$,
\end{quote}
where $\unitvec{i}$ denotes the $i$th unit vector
for $i=1,2,\ldots, n$.
It is known that the Minkowski sum of \Mnat-convex sets
is \Mnat-convex
(\cite[Theorem 6.15]{Mdcasiam},
\cite[Theorem 3.13]{Msurvop21}).
The following theorem states this fact.

\begin{theorem}  \label{THminkowM}
The Minkowski sum 
$W = S_{1} + S_{2} + \cdots + S_{m}$
of \Mnat-convex sets 
$S_{i} \subseteq \ZZ\sp{n}$ $(i=1,2,\ldots,m)$
is an \Mnat-convex set. 
\finbox
\end{theorem}

\begin{corollary}  \label{COminkowMholefree}
For the Minkowski sum 
$W = S_{1} + S_{2} + \cdots + S_{m}$
of \Mnat-convex sets 
$S_{i} \subseteq \ZZ\sp{n}$ $(i=1,2,\ldots,m)$,
we have $\overline{W} \cap \ZZ\sp{n} = W$.
\end{corollary}
\begin{proof}
$W$ is an \Mnat-convex set by Theorem~\ref{THminkowM}.
Any \Mnat-convex set is an integrally convex set,
for which \eqref{icsetholefree} holds.
\end{proof}

Another subclass of integrally convex sets,
called \Lnat-convex sets,
is defined as follows.
A (nonempty) set $S \subseteq \ZZ\sp{n}$ 
is called an 
{\em \Lnat-convex}
if it enjoys discrete midpoint convexity:
\begin{equation*}  
 x, y \in S
\ \Longrightarrow \
\left\lceil \frac{x+y}{2} \right\rceil ,
\left\lfloor \frac{x+y}{2} \right\rfloor  \in S .
\end{equation*}
The Minkowski sum of two \Lnat-convex sets
is not necessarily \Lnat-convex but it is integrally convex
\cite[Theorem 8.42]{Mdcasiam}.
It is also noted that the Minkowski sum of three \Lnat-convex sets
is no longer integrally convex (Example~\ref{EXminkow3lnatset}).
See \cite[Section 3.5]{Msurvop21}
for the Minkowski sum operation
for other kinds of discrete convex sets.

\section{Results}
\label{SCresult}

In this section we present our main results, 
the Shapley--Folkman-type theorems for 
integrally convex sets, \Mnat-convex sets, and \Lnat-convex sets.
To state the theorems we need to define
functions
\begin{equation} \label{ShFoStBnddef}
\alpha(n,m) =   \left( 1-\frac{1}{n} \right) \min (n,m) ,
\qquad
\beta(n,m) = \frac{1}{2} \sqrt{ n \cdot \min (n,m) } ,
\end{equation}
where $n$ is the dimension of the space and
$m$ is the number of Minkowski summands.
The proofs are given in Section~\ref{SCproof}.

\begin{theorem}  \label{THshapfolkLinfIC}
Let $S_{i} \subseteq \ZZ\sp{n}$ 
$(i=1,2,\ldots,m)$
be integrally convex sets and 
$W = S_{1} + S_{2} + \cdots + S_{m}$,
where $n \geq 2$.
For any $x \in \overline{W}$,
there exists $z \in W$ that satisfies 
$\| x - z \|_{\infty} \leq \alpha(n,m)$.
If $x \in \overline{W} \cap \ZZ\sp{n}$,
in particular, then
$\| x - z \|_{\infty} \leq \lfloor \alpha(n,m) \rfloor = \min (n,m)-1$.
\finbox
\end{theorem}

\begin{theorem}  \label{THshapfolkL2IC}
Let $S_{i} \subseteq \ZZ\sp{n}$ 
$(i=1,2,\ldots,m)$
be integrally convex sets
and 
$W = S_{1} + S_{2} + \cdots + S_{m}$,
where $n \geq 1$.
For any $x \in \overline{W}$,
there exists $z \in W$
that satisfies 
$\| x - z \|_{2} \leq \beta(n,m)$
(and hence $\| x - z \|_{\infty} \leq \beta(n,m)$).
If $x \in \overline{W} \cap \ZZ\sp{n}$,
in particular, then
$\| x - z \|_{\infty} \leq \lfloor \beta(n,m) \rfloor$.
\finbox
\end{theorem}

\begin{example} \rm \label{EXshfoIC2}
In Figure~\ref{FGminkowhole} (Example~\ref{EXicdim2sumhole}),
we have $n=2$, $m=2$, $\alpha(n,m)=\beta(n,m) = 1$.
For $x=(1,1) \in \overline{S_{1}+ S_{2}}$,
which is a `hole,' we can take $z=(1,0)\in S_{1}+ S_{2}$
satisfying  $\| x - z \|_{\infty} \leq 1$.
\finbox
\end{example}

A combination of Theorems \ref{THshapfolkLinfIC} and \ref{THshapfolkL2IC}
implies that, for any $x \in \overline{W}$,
there exists $z \in W$ that satisfies 
\begin{equation} \label{shapfolkLinfICmin}
 \| x - z \|_{\infty} \leq \min \{ \alpha(n,m), \beta(n,m) \} 
\qquad (n \geq 2, m \geq 1);
\end{equation}
if $x \in \overline{W} \cap \ZZ\sp{n}$,
in particular, then
\begin{equation} \label{shapfolkLinfICminZ}
 \| x - z \|_{\infty} \leq 
\min \{ \lfloor \alpha(n,m) \rfloor, \lfloor \beta(n,m) \rfloor \}
\qquad (n \geq 2, m \geq 1).
\end{equation}

The following proposition determines which is smaller between 
$\alpha(n,m)$ and $\beta(n,m)$ depending on $(n,m)$.
The proof is given in Section~\ref{SCproofalphabeta}.
Roughly speaking, 
$\alpha(n,m)$ is smaller when $m$ is small,
and $\beta(n,m)$ is smaller when $m$ is large.

\begin{proposition} \label{PRmnthreshV2}  
\quad

\noindent
{\rm (1)} 
Case of $n=2$: \ 
$\alpha(2,m)=\beta(2,m)=1$ \ for all \  $m \geq 2$.

\noindent
{\rm (2)} 
Case of $m=1$: \ 
$\alpha(n,1)< \beta(n,1)$ \  for all \   $n \geq 2$.

\noindent
{\rm (3)} 
Case of $m \geq 2$:  \ 
$\alpha(n,m) >  \beta(n,m)$ if $3 \leq n \leq 4m -3$,
and
$\alpha(n,m)< \beta(n,m)$  if $n \geq 4m -2$.
\finbox
\end{proposition}

The values of $\lfloor \alpha(n,m) \rfloor$ and $\lfloor \beta(n,m) \rfloor$
used in \eqref{shapfolkLinfICminZ} for an integral point $x$
are shown below.
For each $(n,m)$, the smaller of the two is in  boldface.
\[
\begin{array}{l|cc|cc|cc|cc|cc|cc|}
 & \multicolumn{2}{c|}{m=1}  
 & \multicolumn{2}{c|}{m=2}  
  & \multicolumn{2}{c|}{m=3}  
  & \multicolumn{2}{c|}{m=4} 
  & \multicolumn{2}{c|}{m=5} 
\\
  & \lfloor \alpha \rfloor & \lfloor \beta \rfloor 
  & \lfloor \alpha \rfloor & \lfloor \beta \rfloor 
  & \lfloor \alpha \rfloor & \lfloor \beta \rfloor 
  & \lfloor \alpha \rfloor & \lfloor \beta \rfloor 
  & \lfloor \alpha \rfloor & \lfloor \beta \rfloor 
\\ \hline
n=2 & 0 & 0 & 1 & 1 & 1 & 1 & 1 & 1 & 1 & 1 
\\
n=3 & 0 & 0 & 1 & 1 & 2 & {\bf 1} & 2 & {\bf 1} & 2 & {\bf 1} 
\\
n=4 & {\bf 0} & 1 & 1 & 1 & 2 & {\bf 1} & 3 & {\bf 2} & 3 & {\bf 2} 
\\
n=8 & {\bf 0} & 1 & {\bf 1} & 2 & 2 & 2 & 3 & {\bf 2} & 4 & {\bf 3}  
\\
n=12 & {\bf 0} & 1 & {\bf 1} & 2 & {\bf 2} & 3 & 3 & 3 & 4 & {\bf 3} 
\\
n=16 & {\bf 0} & 2 & {\bf 1} & 2 & {\bf 2} & 3 & {\bf 3} & 4 & 4 & 4  
\\ \hline
\end{array}
\]

The particular case of Theorem~\ref{THshapfolkLinfIC} 
for $m=1$
is worthy of attention.
For $m=1$, we have 
$\alpha(n,1) = 1 - 1/n$ for $n \geq 2$,
and hence
$\lfloor \alpha(n,1) \rfloor = 0$ for all $n \geq 2$.
The latter 
(i.e., $\lfloor \alpha(n,1) \rfloor = 0$)
corresponds to the fact that 
$S = \overline{S} \cap \ZZ\sp{n}$
for an integrally convex set $S$.
A combination of the former 
(i.e., $\alpha(n,1) = 1 - 1/n$)
with Theorem~\ref{THminkowM}
results in a sharp bound 
for the case of \Mnat-convex summands $S_{i}$.

\begin{theorem}  \label{THshapfolkM}
Let $S_{i} \subseteq \ZZ\sp{n}$ $(i=1,2,\ldots,m)$
be \Mnat-convex sets and 
$W = S_{1} + S_{2} + \cdots + S_{m}$,
where $n \geq 2$.
For any $x \in \overline{W}$,
there exists $z \in W$ that satisfies 
$\| x - z \|_{\infty} \leq 1 -  1/n$.
\end{theorem}
\begin{proof}
Since the Minkowski sum of 
\Mnat-convex sets remains to be \Mnat-convex (Theorem~\ref{THminkowM}),
$W$ is an \Mnat-convex set, and hence it is an integrally convex set.
By Theorem~\ref{THshapfolkLinfIC} with $m=1$,
there exists $z \in W$
that satisfies 
$\| x - z \|_{\infty} \leq \alpha(n,1) = 1 - 1/n$.
\end{proof}

For \Lnat-convex summands $S_{i}$
we can derive the following bounds immediately from 
Theorems \ref{THshapfolkLinfIC} and \ref{THshapfolkL2IC}.

\begin{theorem}  \label{THshapfolkL}
Let $S_{i} \subseteq \ZZ\sp{n}$  $(i=1,2,\ldots,m)$
be \Lnat-convex sets,
$W = S_{1} + S_{2} + \cdots + S_{m}$,
and $x \in \overline{W}$, where $n \geq 2$.

\noindent
{\rm (1)} 
There exists $z \in W$ that satisfies 
$\| x - z \|_{\infty} \leq \alpha(n,\lceil m/2 \rceil)$.
If $x \in \overline{W} \cap \ZZ\sp{n}$, then
$\| x - z \|_{\infty} \leq \lfloor \alpha(n,\lceil m/2 \rceil) \rfloor 
= \min (n,\lceil m/2 \rceil)-1$.

\noindent
{\rm (2)} 
There exists $z \in W$ that satisfies 
$\| x - z \|_{2} \leq \beta(n,\lceil m/2 \rceil)$.
If $x \in \overline{W} \cap \ZZ\sp{n}$, then
$\| x - z \|_{\infty} \leq \lfloor \beta(n,\lceil m/2 \rceil) \rfloor$.

\noindent
{\rm (3)} 
There exists $z \in W$ that satisfies 
$ \| x - z \|_{\infty} \leq 
\min \{ \alpha(n,\lceil m/2 \rceil), \beta(n,\lceil m/2 \rceil) \}$. 
If $x \in \overline{W} \cap \ZZ\sp{n}$,
then
$ \| x - z \|_{\infty} \leq 
\min \{ \lfloor \alpha(n,\lceil m/2 \rceil) \rfloor, 
        \lfloor \beta(n,\lceil m/2 \rceil) \rfloor \}$.
\end{theorem}
\begin{proof}
By \cite[Theorem 8.42]{Mdcasiam},
the Minkowski sum of two \Lnat-convex sets is integrally convex
(though not \Lnat-convex).
This allows us to replace $m$ by
$\lceil m/2 \rceil$
in the upper bounds in 
Theorems \ref{THshapfolkLinfIC} and \ref{THshapfolkL2IC}.
\end{proof}

\begin{example}[{\cite[Example 4.12]{MS01rel}}]  \rm \label{EXminkow3lnatset}
Consider three \Lnat-convex sets
$S_{1}  =   \{(0, 0, 0),   \allowbreak    (1, 1, 0)\}$, 
$S_{2}  =  \{(0, 0, 0), (0, 1, 1)\}$, and
$S_{3}  =  \{(0, 0, 0), (1, 0, 1)\}$.
Their Minkowski sum 
$S = S_{1} + S_{2} + S_{3}$
is given by
$S = \{(0,0,0),  \allowbreak  (0,1,1),  \allowbreak   
     (1,1,0),  \allowbreak   (1,0,1),  \allowbreak   (2,1,1),(1,1,2),(1,2,1),(2,2,2)\}$,
which has a hole  $(1,1,1) \in \overline{S} \setminus S$.
Theorem \ref{THshapfolkLinfIC} (with $n=3$, $m=3$)
gives a bound $\| x - z \|_{\infty} \leq 2$,
whereas a better bound
$\| x - z \|_{\infty} \leq 1$
is obtained from  Theorem \ref{THshapfolkL} (1).
\finbox
\end{example}

\begin{remark} \rm \label{RMvohra}
A recent paper by Nguyen--Vohra \cite{NV24}
gives an interesting variant of 
the Shapley--Folkman theorem.
A polytope $P$ with vertices in $\{ 0, 1 \}\sp{n}$ 
is called 
{\em $\Delta$-uniform} if
each of its edges, which is a $\{-1, 0, 1 \}$ vector, 
has at most $\Delta$ positive and at most $\Delta$ negative coordinates.
The theorem of Nguyen and Vohra 
(to be called ``Theorem~NV" here) implies the following:
Let $S_{i}$ $(i=1,2,\ldots,m)$
be subsets of 
$\{ 0, 1 \}\sp{n}$ 
such that each $\overline{S_{i}}$ is $\Delta$-uniform,
and let $W = S_{1} + S_{2} + \cdots + S_{m}$.
For any $x \in \overline{W}$,
there exists $z \in W$ that satisfies 
$\| x - z \|_{\infty} < 2 \Delta -1$.
If $x \in \overline{W} \cap \ZZ\sp{n}$, 
then $\| x - z \|_{\infty} \leq 2\Delta - 2$.

The following comparisons may be made
between Theorem~NV and our results.

\begin{itemize}

\item
Theorem~NV captures a property of the summand sets $S_{i}$ 
$\subseteq \{ 0, 1 \}\sp{n}$
in terms of a parameter $\Delta$ related to edge vectors,
and gives a bound 
on $\| x - z \|_{\infty}$
using $\Delta$, independent of $n$ and $m$.
In contrast,
Theorem~\ref{THshapfolkLinfIC}
exploits no specific properties.
Recall that any subset of
$\{ 0, 1 \}\sp{n}$ is integrally convex.

\item
For arbitrary summand sets
$S_{i} \subseteq \{ 0, 1 \}\sp{n}$,
we can take $\Delta = n$.
For $x \in \overline{W}$,
Theorem~NV gives $\| x - z \|_{\infty} < 2n-1$,
whereas Theorem~\ref{THshapfolkLinfIC} gives
$\| x - z \|_{\infty} \leq \alpha(n,m) = (1 - 1/n) \min (n,m)$.
We have $2n-1 > \alpha(n,m)$
for all $n \ge 2$ and $m \ge 1$.

\item
When each summand $S_{i}$ 
is an \Mnat-convex set contained in $\{ 0,1 \}\sp{n}$
(e.g., arising from the independent sets of a matroid),
we have $\Delta = 1$
and Theorem~NV gives 
$\| x - z \|_{\infty} <  1$
for $x \in \overline{W}$
and $\| x - z \|_{\infty} = 0$
for $x \in \overline{W} \cap \ZZ\sp{n}$,
whereas Theorem~\ref{THshapfolkM} gives
$\| x - z \|_{\infty} \le  1 - 1/n$
for $x \in \overline{W}$ (when $n \ge 2$)
and Corollary~\ref{COminkowMholefree} shows
$\| x - z \|_{\infty} = 0$
for $x \in \overline{W} \cap \ZZ\sp{n}$.

\item
When each summand $S_{i}$ 
is an \Lnat-convex set contained in $\{ 0,1 \}\sp{n}$,
we have 
$\Delta = n$.
For $x \in \overline{W}$,
Theorem~NV gives $\| x - z \|_{\infty} < 2n-1$,
whereas Theorem~\ref{THshapfolkL}(1) gives
$\| x - z \|_{\infty} \leq \alpha(n,\lceil m/2 \rceil)$ (when $n \ge 2$).
We have $2n-1 > \alpha(n,\lceil m/2 \rceil)$
for all $n \ge 2$ and $m \ge 1$.
\finbox
\end{itemize}
\end{remark}

\section{Proofs}
\label{SCproof}

\subsection{Proofs of Theorems \ref{THshapfolkLinfIC} and \ref{THshapfolkL2IC}}
\label{SCproofmainthm}

In this section we prove the main theorems
(Theorems \ref{THshapfolkLinfIC} and \ref{THshapfolkL2IC})
of this paper.
For the proof of Theorem~\ref{THshapfolkLinfIC},
we need the following lemma concerning a subset of $\{ 0,1 \}\sp{n}$ in general,
which may be useful in some other contexts.

\begin{lemma}  \label{LMptinunitcube}
Let $S \subseteq \{ 0,1 \}\sp{n}$, where $n \geq 2$.
For any $x \in \overline{S}$,
there exists $v\sp{*} \in S$ that satisfies 
\begin{equation} \label{ptinunitcube}
 \| x - v\sp{*}  \|_{\infty} \leq 1 - \frac{1}{n}.
\end{equation}
\end{lemma}  
\begin{proof}
The proof is given in Section~\ref{SCproofunitcube}.
\end{proof}

\begin{remark} \rm \label{RMtightness}
The bound 
$ \| x - v\sp{*}  \|_{\infty} \leq 1 - 1/n$
in Lemma~\ref{LMptinunitcube} 
is tight.
For example, for
$S = \{ \unitvec{i} \mid i =1,2,\ldots, n \}
 = \{ 
(1,0,0,\ldots,0,0),
(0,1,0,\ldots,0,0), \ldots,   \allowbreak  (0,0,0,\ldots,0,1) \}$
and $x = (1/n, 1/n, \ldots, 1/n) \in \overline{S}$,
we have 
$ \| x - v  \|_{\infty} = 1 - 1/n$ for all $ v \in S$.
\finbox
\end{remark}

We can prove Theorem~\ref{THshapfolkLinfIC} as follows.
Since
\[
x \in \overline{S_{1}+ S_{2} + \cdots + S_{m}}
 = \overline{S_{1}} + \overline{S_{2}} + \cdots + \overline{S_{m}},
\]
the vector $x$ can be represented 
as a sum of some elements of 
$\overline{S_{1}}, \overline{S_{2}}, \ldots ,\overline{S_{m}}$.
That is,
\begin{equation} \label{icSFprf1}
 x = \sum_{i=1}\sp{m} y\sp{i} 
\end{equation}
for some
$y\sp{i} \in \overline{S_{i}}$
$(i=1,2,\ldots,m)$.
Let 
\begin{equation} \label{icSFprf3defTi}
T_{i} = S_{i} \cap {\rm N}(y\sp{i})
\end{equation}
for $i=1,2,\ldots,m$,
where ${\rm N}(y\sp{i})$ is the integral neighborhood of $y\sp{i}$ 
defined in \eqref{Nxdef}.
Since each $S_{i}$ is integrally convex,
we may use \eqref{icsetdef1} to obtain
$y\sp{i} \in \overline{S_{i} \cap {\rm N}(y\sp{i})} = \overline{T_{i}}$.
Then \eqref{icSFprf1} shows
$x \in \overline{T_{1}+ T_{2} + \cdots + T_{m}}$.

By Theorem~\ref{THshapfolkG0}
(Shapley--Folkman's lemma)
there exists
$I \subseteq \{ 1,2,\ldots,m \}$
such that
$|I| \leq \min(n,m)$ and
$x \in \overline{\sum_{i \in I} T_{i}} + \sum_{j \in J} T_{j}$,
where $J = \{ 1,2,\ldots,m \} \setminus I$.
Therefore, 
\begin{equation*} 
x = \sum_{i \in I} x\sp{i} + \sum_{j \in J} z\sp{j}
\end{equation*}
for some
$x\sp{i} \in \overline{T_{i}}$ $(i \in I)$ and
$z\sp{j} \in T_{j}$ $(j \in J)$.
Lemma~\ref{LMptinunitcube} implies that, for each $i \in I$,
there exists
$v\sp{i} \in T_{i}$
satisfying
$\| x\sp{i} - v\sp{i} \|_{\infty} \leq 1-1/n$.
Define 
\[
 z = \sum_{i \in I} v\sp{i} + \sum_{j \in J} z\sp{j},
\]
which belongs to 
$T_{1}+ T_{2} + \cdots + T_{m}$  $(\subseteq S_{1}+ S_{2} + \cdots + S_{m} = W)$.
We then have
\[ 
 \| x - z \|_{\infty}  
 = \| \sum_{i \in I} (x\sp{i} - v\sp{i}) \|_{\infty} 
 \leq \sum_{i \in I}  \| x\sp{i} - v\sp{i} \|_{\infty} 
  \leq  \left( 1 - \frac{1}{n} \right) |I|  \leq \alpha(n,m) .
\] 
Finally, if $x \in \overline{W} \cap \ZZ\sp{n}$,
we have
$\ZZ \ni \| x - z \|_{\infty} \leq \alpha(n,m)$,
whereas 
$\lfloor \alpha(n,m) \rfloor = \min (n,m)-1$.
This completes the proof of Theorem~\ref{THshapfolkLinfIC}.

\medskip

The proof of Theorem~\ref{THshapfolkL2IC} is as follows.
Each $T_{i}$ in \eqref{icSFprf3defTi}
is contained in a translated unit cube, that is,
$T_{i} \subseteq a\sp{i} + \{ 0,1 \}\sp{n}$
for some $a\sp{i} \in \ZZ\sp{n}$,
from which follows that
$r(T_{i}) \leq {\rm rad}(T_{i}) \leq \sqrt{n}/2$
for $i=1,2,\ldots,m$.
Hence we can take
$L=\sqrt{n}/2$ in Theorem~\ref{THshfostarr}
(Shapley--Folkman--Starr theorem),  to obtain
\[
\| x - z \|_{2} \leq L \sqrt{\min (n,m)}
=(\sqrt{n}/2 ) \sqrt{\min (n,m)} = \beta(n,m).
\]
Finally, if $x \in \overline{W} \cap \ZZ\sp{n}$,
we have 
$\ZZ \ni \| x - z \|_{\infty} \leq \| x - z \|_{2} \leq \beta(n,m)$,
from which 
$\| x - z \|_{\infty} \leq \lfloor \beta(n,m) \rfloor$.
Thus Theorem~\ref{THshapfolkL2IC} is proved.

\subsection{Proof of Lemma~\ref{LMptinunitcube}}
\label{SCproofunitcube}

In this section we prove  Lemma~\ref{LMptinunitcube},
which states that 
for any $x \in \overline{S}$,
there exists $v\sp{*} \in S$ satisfying
$ \| x - v\sp{*}  \|_{\infty} \leq 1 - 1/n$
in \eqref{ptinunitcube}.
Let $N = \{ 1,2,\ldots,n \}$.
Without loss of generality, 
we may assume
that $x_{i} \geq 1/2$ for all $i \in N$.
(If $I = \{ i \in N \mid x_{i} < 1/2 \}$ is nonempty, 
change $x_{i}$ to $1-x_{i}$ for all $i \in I$,
and change $S$ similarly.)  
Represent $x$ as a convex combination
of the points of $S$ as  
$x = \sum_{u \in S} \lambda_{u} u$,
where
$\sum_{u \in S} \lambda_{u} =1$ and 
$\lambda_{u} \geq 0$ $(u \in S)$.
We first note the following fact.

Claim 1:  
If $\lambda_{v} \geq 1/n$ for some $v \in S$, 
then $\| x - v  \|_{\infty} \leq 1 - 1/n$ for such $v$.

\noindent
(Proof of Claim 1) 
Since
\[
  x -v  = \sum_{u \in S} \lambda_{u} (u - v) = \sum_{u \ne v} \lambda_{u} (u - v),
\]
we obtain 
\begin{align*}
  \| x - v  \|_{\infty}  &= 
\max_{i \in N} 
 \left\{  \big| \sum_{u \ne v} \lambda_{u} (u_{i} - v_{i}) \big| \right\}
 \leq 
\max_{i \in N} 
  \left\{  \sum_{u \ne v} \lambda_{u} | u_{i} - v_{i} | \right\}
\\
 &\leq  \sum_{u \ne v} \lambda_{u} =  1  - \lambda_{v} \leq 1 - \frac{1}{n}.
\end{align*}
\hfill (End of proof of Claim 1)

\medskip

To prove \eqref{ptinunitcube} by contradiction,
we assume
\begin{equation} \label{ptUCmemo3}
 \| x - v  \|_{\infty} > 1 - \frac{1}{n}
\quad \mbox{for all $v \in S$}.
\end{equation}
We shall derive a contradiction as follows.
We first define a partition of $S$ into two subsets,
$S =  S_{1}\sp{0} \cup S_{1}\sp{1}$,
where $S_{1}\sp{1}$ is nonempty under \eqref{ptUCmemo3}.
Then 
$S_{1}\sp{1}$ is partitioned into $S_{2}\sp{0}$ and $S_{2}\sp{1}$,
where $S_{2}\sp{1}$ is nonempty under \eqref{ptUCmemo3}.
Continuing this way,
we obtain partitions of $S$ of the form
\begin{align*}
 S 
&=  S_{1}\sp{0} \cup S_{1}\sp{1} 
=  S_{1}\sp{0} \cup (S_{2}\sp{0} \cup S_{2}\sp{1})
\\
&=  S_{1}\sp{0} \cup S_{2}\sp{0} \cup (S_{3}\sp{0} \cup S_{3}\sp{1})
= \cdots =
\left( \bigcup_{j=1}\sp{n-1} S_{j}\sp{0} \right) \cup S_{n-1}\sp{1} ,
\end{align*}
where $S_{j-1}\sp{1} = S_{j}\sp{0} \cup S_{j}\sp{1}$ 
and $S_{j}\sp{1} \ne \emptyset$ for each 
$j=1,2,\ldots,n-1$
(with the convention of $S_{0}\sp{1} = S$).
At the final stage, we show that 
$S_{n-1}\sp{1} \ne \emptyset$ leads to a contradiction
to \eqref{ptUCmemo3}.

The first partition $S =  S_{1}\sp{0} \cup S_{1}\sp{1}$ is defined as follows.
By \eqref{ptUCmemo3}
there exists $i_{1} \in N$ and 
$u \in S$
satisfying 
$| x_{i_1} - u_{i_1}| > 1 - 1 / n$,
where $u_{i_{1}} =0$
 since $x_{i_1} \geq 1/2$ by our assumption.
Thus we have
\begin{equation} \label{ptUCmemo4}
  x_{i_1} > 1 - \frac{1}{n} .
\end{equation}
With reference to the component $i_{1}$, 
we classify the vectors in $S$ into two subsets:
\begin{equation} \label{ptUCmemo45}
 S_{1}\sp{0} = \{ v \in S \mid v_{i_1}=0 \},
\quad
 S_{1}\sp{1} = \{ v \in S \mid v_{i_1}=1 \}.
\end{equation}
Since
$x_{i_1} = \sum_{v \in S_{1}\sp{1}} \lambda_{v}$,
it follows from \eqref{ptUCmemo4} that
\begin{equation} \label{ptUCmemo6}
 \sum_{v \in S_{1}\sp{1}} \lambda_{v} > 1 - \frac{1}{n},
\qquad
 \sum_{v \in S_{1}\sp{0}} \lambda_{v} < \frac{1}{n} .
\end{equation}
In particular, $S_{1}\sp{1}  \ne \emptyset$.
It also follows from \eqref{ptUCmemo4} that
\begin{equation} \label{ptUCmemo5}
\mbox{For every $v \in S_{1}\sp{1}$: \quad}
 | x_{i_1} - v_{i_1} | = 1 - x_{i_1} < \frac{1}{n} \leq 1 - \frac{1}{n},
\end{equation}
where $n \geq 2$ is used.   Let $S_{0}\sp{1} = S$.

Claim 2: 
For $j=1,2,\ldots,n-1$, we can choose an index
$i_{j} \in N \setminus \{ i_{1}, i_{2},  \ldots, i_{j-1}  \}$
which defines a partition of $S_{j-1}\sp{1}$ into two parts 
\begin{equation} \label{ptUCmemo7}
 S_{j}\sp{0} = \{ v \in S_{j-1}\sp{1} \mid v_{i_{j}}=0 \},
\quad
 S_{j}\sp{1} = \{ v \in S_{j-1}\sp{1} \mid v_{i_{j}}=1 \}
\end{equation}
such that
\begin{align}
&  x_{i_{j}} > 1 - \frac{1}{n} ,
\label{ptUCmemo8} \\
& 
\mbox{For every $v \in S_{j}\sp{1}$: \quad}
| x_{i_{j}} - v_{i_{j}} | = 1 - x_{i_{j}}  \leq 1 - \frac{1}{n} ,
\label{ptUCmemo9} \\
&
 \sum_{v \in S_{j}\sp{1}} \lambda_{v}
 > 1 - \frac{j}{n},
\qquad
 \sum_{v \in S_{j}\sp{0}} \lambda_{v} < \frac{1}{n} .
\label{ptUCmemo10} 
\end{align}

\noindent
(Proof of Claim 2) 
For $j=1$ we have
\eqref{ptUCmemo7}--\eqref{ptUCmemo10} 
from \eqref{ptUCmemo4}--\eqref{ptUCmemo5}. 
Assuming we have chosen $i_{1}, i_{2}, \ldots, i_{j}$
(where $j < n-1$) 
satisfying  \eqref{ptUCmemo7}--\eqref{ptUCmemo10},
we choose the next index $i_{j+1}$ as follows.
For each $v \in S_{j}\sp{1}$ we have 
$| x_{i_k} - v_{i_k} | \leq 1 - 1/n$
for $k=1,2,\ldots,j$ by \eqref{ptUCmemo9}
while
$\| x - v  \|_{\infty} > 1 - 1/n$
by \eqref{ptUCmemo3}.
Hence there exists $i_{j+1} \in N \setminus \{ i_{1}, i_{2},  \ldots, i_{j}  \}$
and $u \in S_{j}\sp{1}$ satisfying 
$| x_{i_{j+1}} - u_{i_{j+1}} | > 1 - 1 / n$,
where $u_{i_{j+1}} =0$
since $x_{i_{j+1}} \geq 1/2$ by our assumption.
Thus we obtain
\begin{equation} \label{ptUCmemo11}
  x_{i_{j+1}} > 1 - \frac{1}{n} ,
\end{equation}
which is \eqref{ptUCmemo8} for $j+1$.
With the use of this $i_{j+1}$ we define 
a partition
$S_{j}\sp{1} = S_{j+1}\sp{0} \cup S_{j+1}\sp{1}$
by \eqref{ptUCmemo7} for $j+1$.
Then
$S = (S_{1}\sp{0} \cup \cdots \cup S_{j}\sp{0}) 
     \cup (S_{j+1}\sp{0} \cup S_{j+1}\sp{1})$
and
\begin{align}
1 - \frac{1}{n} < x_{i_{j+1}} 
& =  \sum_{v \in S_{j+1}\sp{1}} \lambda_{v}
 + \sum_{k=1}\sp{j} \sum_{v \in S_{k}\sp{0}} \lambda_{v} v_{i_{j+1}}
\nonumber \\ &
 \leq \sum_{v \in S_{j+1}\sp{1}} \lambda_{v}
 + \sum_{k=1}\sp{j} \sum_{v \in S_{k}\sp{0}} \lambda_{v} 
 \label{ptUCmemo11B}
\\ &
= 1 -  \sum_{v \in S_{j+1}\sp{0}} \lambda_{v} .
 \label{ptUCmemo11C}
\end{align}
The second inequality of \eqref{ptUCmemo10} for $j+1$ follows
from \eqref{ptUCmemo11C}.
In \eqref{ptUCmemo11B} we have
$\sum_{v \in S_{k}\sp{0}} \lambda_{v} \leq 1/n$
for $k=1,2,\ldots,j$
by the second inequality of \eqref{ptUCmemo10},
and therefore,
\[
1 - \frac{1}{n}  < \sum_{v \in S_{j+1}\sp{1}} \lambda_{v} + \frac{j}{n}.
\]
Thus we obtain
\[
\sum_{v \in S_{j+1}\sp{1}} \lambda_{v} > 1 - \frac{j+1}{n},
\]
which is the first inequality of \eqref{ptUCmemo10} for $j+1$.
For every $v \in S_{j+1}\sp{1}$
we have \eqref{ptUCmemo11} and $v_{i_{j+1}}=1$, from which we obtain
\[
| x_{i_{j+1}} - v_{i_{j+1}} | = 1 - x_{i_{j+1}} < \frac{1}{n} \leq 1 - \frac{1}{n} ,
\]
showing \eqref{ptUCmemo9} for $j+1$.
\hfill (End of proof of Claim 2) 

\medskip

By \eqref{ptUCmemo10} for $j=n-1$, we have
$S_{n-1}\sp{1} \ne \emptyset$.
Since $S_{n-1}\sp{1} \subseteq S_{j}\sp{1}$
for all $j \leq n-1$, 
any $v \in S_{n-1}\sp{1}$ has the property that
$v_{i_{k}} = 1$ for $k=1,2,\ldots,n-1$,
and $v_{i_{n}} \in \{ 0,1 \}$.
If $S_{n-1}\sp{1}$ contains 
$v\sp{*} = (1,1, \ldots, 1)$, this vector satisfies
$\| x - v\sp{*}  \|_{\infty} \leq 1 - 1/n$,
since
\[
 | x_{i_{j}} - v_{i_{j}}\sp{*} | = 1 - x_{i_{j}} \leq 1 - \frac{1}{n} 
\qquad (j=1,2,\ldots,n-1)
\]
by \eqref{ptUCmemo9} and  
\[
 | x_{i_{n}} - v_{i_{n}}\sp{*} | = 1 - x_{i_{n}}  \leq \frac{1}{2} \leq  1 - \frac{1}{n}.
\]
This contradicts \eqref{ptUCmemo3}.
Otherwise, 
$S_{n-1}\sp{1}$ consists of a unique element 
$u\sp{*}$ with $u_{i_{n}}\sp{*}=0$
and $u_{i}\sp{*}=1$ for $i \ne i_{n}$.
By the first inequality of \eqref{ptUCmemo10} for $j=n-1$ we have
$\lambda_{u\sp{*}} > 1 - (n-1)/n = 1/n$,
which, by Claim 1, implies
$\| x - u\sp{*}  \|_{\infty} \leq 1 - 1/n$,
which is also a contradiction to \eqref{ptUCmemo3}.
The proof of 
Lemma~\ref{LMptinunitcube} is thus completed.

\subsection{Proof of Proposition~\ref{PRmnthreshV2}}
\label{SCproofalphabeta}

In this section we prove Proposition~\ref{PRmnthreshV2}
to determine which is smaller between $\alpha(n,m)$ and $\beta(n,m)$.

(1) 
When $n=2$ and $m \geq 2$, we have
\[
\alpha(2,m) =  \left( 1-\frac{1}{2}\right) \min (2,m) = 1,
\quad
\beta(2,m) =  \frac{1}{2} \sqrt{ 2 \cdot \min (2,m) }=1.
\]

(2)
When  $m=1$ and $n \geq 2$, we have
\[
\alpha(n,1) =  \left( 1-\frac{1}{n}\right) \min (n,1) =   1-\frac{1}{n},
\quad
\beta(n,1) =  \frac{1}{2} \sqrt{ n \cdot \min (n,1) } =  \frac{1}{2} \sqrt{ n }.
\]
When  
$n=2$, we have $\alpha(2,1) = 1/2$, \ $\beta(2,1) = \sqrt{ 2 }/2 = 0.7...$,
and hence
$\alpha(2,1) < \beta(2,1)$.
When  
$n=3$, we have $\alpha(3,1) = 2/3$, \ $\beta(3,1) = \sqrt{ 3 }/2 = 0.86...$,
and hence
$\alpha(3,1) < \beta(3,1)$.
When  
$n \geq 4$,
we have
$\alpha(n,1) < 1$, \  $\beta(n,1) =  \frac{1}{2} \sqrt{ n } \geq 1$,
and hence
$\alpha(n,1) < \beta(n,1)$.

(3)
The claim is concerned with the cases with $m \geq 2$ and $n \geq 3$.
The combination of Case A and Case B below covers all such cases.

Case A: When $n \geq 3$ and $n \leq m$, we have
\[
\alpha(n,m) =  \left( 1-\frac{1}{n}\right) n = n-1 ,
\quad
\beta(n,m) =  \frac{1}{2} \sqrt{ n \cdot n } = \frac{n}{2}.
\]
Therefore, $\alpha(n,m) > \beta(n,m)$.

Case B: When $n \geq 3$, \ $m \geq 2$, and $m < n$, we have
\[
\alpha(n,m) =  \left( 1-\frac{1}{n}\right) m ,
\qquad
\beta(n,m) =  \frac{1}{2} \sqrt{ n \cdot m } .
\]
Therefore, we have
\begin{align}
& 
\alpha < \beta
\iffS
  \left( 1-\frac{1}{n}\right) m \
 <  \frac{1}{2} \sqrt{ n \cdot m } 
\iffS
\sqrt{ m } < \frac{\sqrt{ n }}{2}  \frac{1}{ 1- 1/n } 
\iffS
 m  < \frac{n\sp{3}}{ 4 (n-1)\sp{2} } .
\end{align}
Define
\begin{equation} \label{thetadef}
 \theta(n)= \frac{n\sp{3}}{ 4 (n-1)\sp{2} } .
\end{equation}
Since 
$\theta(n)$ is not an integer for any integer $n \geq 3$,
we have that
$\alpha \ne \beta$ for all $(n,m)$, and that
\begin{equation} \label{alphabetatheta}
\alpha < \beta \iffS m  < \theta(n),
\qquad
\alpha > \beta \iffS m  > \theta(n) .
\end{equation}

Case B-1:
When $n=3$, we have $\theta(3) = 27/16 = 1.6875$, and hence
$\alpha(3,2) > \beta(3,2)$
by \eqref{alphabetatheta}.
Note that $\{ m \in \ZZ \mid m \geq 2,  m < n \}$ consists of $m=2$ only.

Case B-2:
When $n=4$,
we have 
$\theta(4) = 16/9 = 1.77...$,
and hence
$\alpha(4,m) > \beta(4,m)$ for $m=2,3$.
Note that
$\{ m \in \ZZ \mid m \geq 2,  m < n \}$
consists of $m=2,3$ only.

Case B-3:
When $n \geq 5$, the threshold value
$\theta(n)$
can be estimated as
\begin{equation} \label{threshest3}
 \frac{ n+2 }{4} <  \frac{n\sp{3}}{ 4 (n-1)\sp{2} } 
 < \frac{ n+3 }{4} 
\qquad(n \geq 5) .
\end{equation}
Indeed, the first inequality of \eqref{threshest3} holds since
\[
 \frac{ n+2 }{4} < \frac{n\sp{3}}{ 4 (n-1)\sp{2} }
\iffS (n+2)(n-1)\sp{2} < n\sp{3} \iffS 3n > 2 ,
\]
and the second inequality of \eqref{threshest3} follows from
\[
\frac{n\sp{3}}{ 4 (n-1)\sp{2} } < \frac{ n+3 }{4}
\iffS  n\sp{3} < (n+3)(n-1)\sp{2}   
\iffS n\sp{2}-5n + 3 > 0
\]
and
$n\sp{2}-5n + 3  = n(n-5) + 3 > 0$.
It follows from 
\eqref{alphabetatheta} and  \eqref{threshest3}
that
\begin{align*} 
& 
\alpha < \beta  \quad \mbox{if} \quad n \geq 5, 2 \leq m \leq (n+2)/4,
\\ &
\alpha > \beta \quad \mbox{if}  \quad n \geq 5, (n+3)/4 \leq  m < n ,
\end{align*}
or equivalently,
\begin{align*} 
& 
\alpha < \beta  \quad \mbox{if} \quad  n \geq 5, 2 \leq m, n \geq 4m-2,
\\ &
\alpha > \beta \quad \mbox{if}  \quad  n \geq 5, 2 \leq m < n \leq 4m-3.
\end{align*}

This completes the proof of Proposition~\ref{PRmnthreshV2}.

\bigskip

\noindent {\bf Acknowledgement}. 
The authors thank Michihiro Kandori for a stimulating question
that triggered this study.
This work was supported by JSPS/MEXT KAKENHI JP23K11001 and JP21H04979,
and by JST ERATO Grant Number JPMJER2301, Japan.
Theorem~\ref{THshapfolkL} 
was obtained from discussion at an ERATO meeting.





\end{document}